\documentclass[11pt, reqno]{amsart}
\usepackage[utf8]{inputenc}
\usepackage{amsmath}
\usepackage{amsthm}
\usepackage{mathtools}
\usepackage{amsfonts}
\usepackage{MnSymbol}
\usepackage{blindtext,amsmath,amsthm,amsfonts, textcomp, mathrsfs, mathtools}
\usepackage[shortlabels]{enumitem}
\DeclarePairedDelimiter \floor {\lfloor} {\rfloor}

\usepackage[top=30mm,bottom=30mm,left=30mm,right=30mm]{geometry}
\newtheorem{theorem}{Theorem}[section]

\newtheorem{corollary}{Corollary}

\newtheorem*{theorem*}{Theorem}
\newtheorem*{remark*}{Remark}
\newtheorem*{problem*}{Problem}
\newtheorem*{conjecture*}{Conjecture}
\newtheorem*{question*}{Question}
\newtheorem{lemma}[theorem]{Lemma}
\newtheorem{proposition}[theorem]{Proposition}

\newcommand{\rS}{\mathcal{S}}

\begin{document}
\title[Prime distribution in certain sequences]{How often are $\floor{n^{\alpha}}$ and $\floor{n^{\beta}}$ simultaneously primes?}

 \author[Anup B. Dixit, Nikhil S. Kumar]{Anup B. Dixit, Nikhil S. Kumar}


\address[Anup B. Dixit]{Institute of Mathematical Sciences, A CI of Homi Bhabha National Institute, 4th cross street, CIT Campus, Taramani, Chennai, 600113}
\email{anupdixit@imsc.res.in}

\address[Nikhil S. Kumar]{Department of Mechanical Engineering, National Institute of Technology, Tiruchirappalli, Tamilnadu, 620015}
\email{nikhil1729kumar@gmail.com}

\date{\today}



\thanks{The research of the first author is partially supported by an Inspire Faculty fellowship. \\ The authors have no conflicts of interest to declare that are relevant to the content of this article.\\ Data sharing is not applicable to this article as no datasets were generated or analysed during the current study.}

\begin{abstract}
Let $\floor{x}$ denote the greatest integer less than or equal to a real number $x$. Given real numbers $0<\alpha_1 < \alpha_2 < \cdots< \alpha_k < 1$ satisfying a certain condition, we show that there are infinitely many positive integers $n$ for which all of $\floor{n^{\alpha_1}}, \floor{n^{\alpha_2}},\ldots, \floor{n^{\alpha_k}}$ are prime numbers. Our approach relies on establishing a simultaneous equidistribution theorem for $\floor{n^{\alpha_i}}$ across $k$-many arithmetic progressions.
\end{abstract}
\maketitle

\section{\bf Introduction}
\bigskip

In his 1912 address at the International Congress of Mathematics, E. Landau proposed four unsolved problems in number theory. The fourth problem was to show that there are infinitely many prime numbers of the form $n^2 +1$, where $n$ is a positive integer. This problem is a special case of the Bunyakovsky conjecture for general polynomials and remains widely open. A natural variant of this problem is to consider the primality of expressions involving real powers of integers. For a real number $x$, let $\floor{x}$ denote the greatest integer $\leq x$. Given a real number $\alpha>0$, one can ask whether  $\floor{n^\alpha}+1$ is a prime infinitely often. When $\alpha=2$, this is precisely Landau's problem. For $0<\alpha\leq 1$, the answer follows from the infinitude of prime numbers. The case $\alpha >1$ was first addressed by Piatetski-Shapiro \cite{piatetski}, who  showed that $\floor{n^{\alpha}}$ is prime infinitely often for all $1<\alpha< \frac{12}{11}$. This range has since been extended through the work of Kolesnik \cite{kolenisk-1} \cite{kolenisk-2}, Heath-Brown \cite{heath-brown}, Liu-Rivat \cite{rivat-1}, with the current best bound $\alpha < \frac{2817}{2426}$ due to Rivat-Sargos \cite{rivat-2}.\\

In this paper, we are motivated by the following natural extension of the above problem: Given real numbers $\alpha,\beta >0$, how often are $\floor{n^{\alpha}}$ and $\floor{n^{\beta}}$ simultaneously primes? More generally, for positive real numbers $\alpha_1, \alpha_2, \cdots, \alpha_k$, how often are all of $\floor{n^{\alpha_1}}, \floor{n^{\alpha_2}},\ldots, \floor{n^{\alpha_k}}$ primes?\\

Let $\Lambda(n)$ be the von Mangoldt function defined as
\begin{equation*}
    \Lambda(n) := \left\{
	\begin{array}{ll}
		\log p  & \mbox{if } n=p^k \text{ for some prime } p\\
		0 & \mbox{otherwise. } 
	\end{array}
\right.
\end{equation*}

Our main theorem is as follows.

\begin{theorem}\label{primes-simultaneous}
    Let $0<\delta=\alpha_1<\alpha_2<\cdots<\alpha_k=\gamma<1$. Suppose 
    \begin{equation*}
        \sum_{i=1}^k \alpha_i \leq \frac{9\delta}{20}-\frac{2\gamma}{5} + \frac{2}{5}.
    \end{equation*} 
    Then, for any integers $c_1,c_2,\ldots, c_k$, as $N$ tends to infinity,
\begin{equation*}
    \sum_{n\leq N} \Lambda(\floor{n^{\alpha_1}}+c_1)\Lambda(\floor{n^{\alpha_2}}+c_2)\cdots\Lambda(\floor{n^{\alpha_k}}+c_k) = N(1 + o(1)).
\end{equation*}

\end{theorem}
\medskip

{\bf Remark.} The case $k=1$ directly follows from the prime number theorem. We illustrate this below. Write
\begin{equation*}
    \sum_{n\leq N} \Lambda(\floor{n^{\alpha}}+c) = \sum_{n\leq N^{\alpha}} \Lambda(n) P(n),
\end{equation*}
where 
\begin{equation*}
    P(n) = \#\{m : \floor{m^{\alpha}} + c = n\}.
\end{equation*}
Note that, if $\floor{m^{\alpha}}+ c=n$, then
\begin{equation*}
    n\leq \floor{m^{\alpha}} + c < n+1,
\end{equation*}
which implies
\begin{equation*}
    (n-c)^{1/\alpha} \leq m < (n-c+1)^{1/\alpha}.
\end{equation*}
Therefore,
\begin{equation*}
    P(n) = \left((n-c+1)^{1/\alpha} - (n-c)^{1/\alpha}\right) + O(1),
\end{equation*}
where the implied constant is $<1$. Recall the prime number theorem \cite{poussin}, which states that as $x$ tends to infinity
\begin{equation*}
     \sum_{n\leq x} \Lambda(n) = x + O\left(xe^{-c_0\sqrt{\log x}}\right).
\end{equation*}
Therefore
\begin{align*}
    \sum_{n\leq N} \Lambda([n^{\alpha}]+c) &= \sum_{n\leq N^{\alpha}} \left((n-c+1)^{1/\alpha} - (n-c)^{1/\alpha}\right) \Lambda(n) + O\left(\sum_{n\leq N^{\alpha}} \Lambda(n) \right)\\
    & = \sum_{n\leq N^{\alpha}} \left((n-c+1)^{1/\alpha} - (n-c)^{1/\alpha}\right) \Lambda(n) + O\left(N^{\alpha}\right).
\end{align*}
Using partial summation,
\begin{align}\label{pnt-application}
    &\sum_{n\leq N^{\alpha}} \left((n-c+1)^{1/\alpha} - (n-c)^{1/\alpha}\right) \Lambda(n) \nonumber\\
    & =\left(\sum_{n\leq N^{\alpha}} \Lambda(n)\right) \left( (N^{\alpha} - c+1)^{1/\alpha} - (N^{\alpha}-c)^{1/\alpha}\right) \nonumber\\
    & \hspace{4cm}- \int_1^{N^\alpha} \left(\sum_{n\leq y} \Lambda(n)\right) \left(\frac{1}{\alpha} \left((y-c+1)^{-1 +1/\alpha } - (y-c)^{-1+1/\alpha}\right)\right) dy.
\end{align}
In the first term, write
\begin{align*}
    (N^{\alpha} - c+1)^{1/\alpha} - (N^{\alpha}-c)^{1/\alpha} &= (N^{\alpha}-c)^{1/\alpha} \left(\left(1+\frac{1}{N^{\alpha}-c}\right)^{1/\alpha} -1\right)\\
    & \sim N \left( 1 + \frac{1}{\alpha N^{\alpha}}-1\right)\sim \frac{N^{1-\alpha}}{\alpha},
\end{align*}
for large values of $N$ using binomial theorem. Now treating the second term in the integrand similarly and applying the prime number theorem, the sum in \eqref{pnt-application} is equal to
\begin{align*}
    &=\frac{N}{\alpha}- \frac{1}{\alpha} \int_1^{N^{\alpha}} \left(\frac{1}{\alpha} - 1 \right) y^{-1 +1/\alpha}\, dy + O\left(N\, e^{-c_0\sqrt{\log N}}\right) + O\left(N^{1-\alpha}\right)\\
    & =\frac{N}{\alpha} + N \left(1-\frac{1}{\alpha}\right) + O\left(N\, e^{-c_0\sqrt{\log N}}\right)\\
    & =N + O\left(N\, e^{-c_0\sqrt{\log N}}\right),
\end{align*}
as required.
\medskip

The case $k=2$ in Theorem \ref{primes-simultaneous} is already non-trivial, giving the following corollary.

\begin{corollary}\label{two-variable}
    For $0<\beta<\alpha< 1$ satisfying $28 \alpha + 11\beta \leq 8$, as $N\to\infty$
    \begin{equation*}
        \sum_{n\leq N} \Lambda(\floor{n^{\alpha}}) \Lambda(\floor{n^{\beta}}) = N (1+  o(1)).
    \end{equation*}
\end{corollary}

This shows that for $\alpha, \beta$ satisfying the above condition, $\floor{n^{\alpha}}$ and $\floor{n^{\beta}}$ are simultaneously primes infinitely often.\\

\noindent
{\bf Remark.} It is worth noting that for certain values of $\alpha$ and $\beta$, it is possible to conclude that $\floor{n^{\alpha}}$ and $\floor{n^{\beta}}$ are both primes infinitely often, using known results on primes in short intervals. Suppose $0<\beta<\alpha<1$ and $\floor{n^{\beta}}= p$, a prime number. Then, we have $\floor{n^{\beta}} = p$ for all $n\in [ p^{1/\beta}, (p+1)^{1/\beta})$. Taking exponent $\alpha$ to this interval gives $[p^{\alpha/\beta}, (p+1)^{\alpha/\beta})$. Note that $(p+1)^{\alpha/\beta} \sim p^{\alpha/\beta} + \frac{\alpha}{\beta} p^{\alpha/\beta -1}$.  We now invoke the best known result on primes in short intervals due to Baker, Harman and Pintz \cite{baker}, which states that for sufficiently large $x$, there exists a prime in the interval $[x, x+ x^{0.525}]$. Therefore, if $\frac{\alpha}{\beta}-1 > 0.525 \frac{\alpha}{\beta}$, we conclude that there exists a prime in the interval $[p^{\alpha/\beta}, (p+1)^{\alpha/\beta})$ and hence there exists an $n\in  [ p^{1/\beta}, (p+1)^{1/\beta})$ such that $\floor{n^{\alpha}}$ and $\floor{n^{\beta}}$ are simultaneously primes. Hence for $0<\beta<\alpha< 1$ satisfying $\beta < 0.475 \alpha$, there are infinitely many $n$ such that $\floor{n^{\alpha}}$ and $\floor{n^{\beta}}$ are both primes. The condition $28 \alpha + 11\beta \leq 8$ in Corollary \ref{two-variable} holds in more generality.

\bigskip

The key ingredient in our approach is to establish an equidistribution result for $\floor{n^{\alpha_1}}, \floor{n^{\alpha_2}},$ $\ldots$, $\floor{n^{\alpha_k}}$ simultaneously across $k$-many arithmetic progressions. This can be thought of as a higher dimensional analog of a result due to M. R. Murty and K. Srinivas \cite{ram}.\\

Let $\mathbf{c} = (c_1,c_2,\ldots, c_k)$ and $\mathbf{d} = (d_1,d_2,\ldots, d_k)$ be tuples of positive integers satisfying $0\leq c_i < d_i$ for all $1\leq i\leq k$. Let $\pmb\alpha= (\alpha_1,\alpha_2,\ldots, \alpha_k)$ be a tuple of positive real numbers. Define
\begin{equation*}
    S(N, \mathbf{d}, \mathbf{c}, \pmb \alpha) := \#\bigg\{ n\leq N : \floor{n^{\alpha_1}} \equiv c_1 \bmod d_1 , \floor{n^{\alpha_2}} \equiv c_2 \bmod d_2 ,\ldots, \floor{n^{\alpha_k}} \equiv c_k \bmod d_k\bigg\}.
\end{equation*}
We prove the following theorem, which is interesting in its own right.

\begin{theorem}\label{multivariable-equidistribution}
    Let $\mathbf{d}, \mathbf{c}, \pmb \alpha$ be as above. Suppose $0<\delta=\alpha_1<\cdots<\alpha_k=\gamma<1$. Then, as $N$ tends to infinity,
\begin{equation*}
 S(N, \mathbf{d}, \mathbf{c}, \pmb \alpha) = \frac{N}{d_1d_2\ldots d_k} +  O\left(\frac{N^{\frac{3 + 2\gamma - \delta}{5}} (\log N)^{k+1}}{(\min(d_1,d_2,\ldots,d_k))^\frac{1}{4}} \right).
\end{equation*}
\end{theorem}
\medskip

\noindent
For $k=2$, this yields the following corollary.

\begin{corollary}
    For positive integers $c_1,c_2,d_1, d_2$ and real numbers $\alpha,\beta$ satisfying $0<\beta<\alpha<1$, as $N$ tends to infinity
    \begin{equation*}
        \#\bigg\{n\leq N : \floor{n^{\alpha}} \equiv c_1 \bmod d_1 \,\, \text{and} \,\, \floor{n^{\beta}} \equiv c_2 \bmod d_2\bigg\} =  \frac{N}{d_1d_2} +  O\left(\frac{N^{\frac{3 + 2\alpha - \beta}{5}} (\log N)^3}{(\min(d_1,d_2))^\frac{1}{4}} \right)
    \end{equation*}
\end{corollary}

The assumption that the exponents $\alpha_i$'s are distinct is essential. Indeed, if two exponents are the same, incompatible congruence conditions could render the set $S(N, \mathbf{d}, \mathbf{c}, \pmb \alpha)$ empty. When $\alpha_i$'s are distinct, the count $S(N, \mathbf{d}, \mathbf{c}, \pmb \alpha)$ exhibits true independence across all arithmetic progressions. This phenomena is precisely captured in Proposition \ref{unif-distribution} in the next section.

\bigskip
 
\section{\bf Preliminaries}
\bigskip

In this section, we establish some lemmata which will be instrumental in proving the main results. For a non-zero real number $h$, let 
\begin{equation*}
    \rS(N) = \sum_{n=1}^N e(h n^{\alpha}),
\end{equation*}
where $e(x) := e^{2\pi i x}$. In \cite[Theorem 1(a)]{ram}, M. R. Murty and K. Srinivas proved that for all integers $h\neq 0$, as $N$ tends to infinity
\begin{equation*}
    \rS(N) = O_{\alpha}\left(|h|^{1/4}\, N^{(\alpha+2)/4} \,(\log N)\, (\log N|h|)\right),
\end{equation*}
for $0<\alpha<2$ and $\alpha \neq 1$. The result in \cite{ram} assumes $h$ to be an integer. However, the same proof goes through without such an assumption. A crucial role in the proof of our theorem involves a higher dimensional analog of the above result as stated below.\\

Let $\mathbf{h}=(h_1,h_2,\ldots, h_k)$ and $\pmb\alpha=(\alpha_1,\alpha_2,\ldots, \alpha_k)$ be tuples of non-zero real numbers. Define
\begin{equation*}
    \rS(N, \mathbf{h}, \pmb\alpha) :=  \sum_{n=1}^N e\left(h_1 n^{\alpha_1} + h_2 n^{\alpha_2}+\cdots + h_k n^{\alpha_k}\right).
\end{equation*}
We obtain upper bounds for $\rS(N, \mathbf{h}, \pmb\alpha)$ as follows.

\begin{proposition}\label{unif-distribution}
    Let $0<\alpha_1,\alpha_2,\ldots,\alpha_k<1$ with $\gamma = \max (\alpha_j)$ and $\delta = \min (\alpha_j)$. Let $h_1,h_2,\ldots, h_k$ be real numbers such that at least one of the $h_i$'s is non-zero. Suppose that for any subset $E \subseteq \{1,2,\dots,k\}$ with $\alpha_i = \alpha_j$ for all $i,j \in E$, either
    \begin{enumerate}
        \item $\sum_{i\in E} h_i \neq 0$, or
        \item there exists an $l \notin E$ such that $h_l \neq 0$.
    \end{enumerate}
   Then, as $N$ tends to infinity, we have
    \begin{equation*}
        \rS(N, \mathbf{h}, \pmb\alpha)= O_{\pmb{\alpha}}\left(\left( |h_1| + |h_2| + \cdots + |h_k| \right)^\frac{1}{4}N^{\frac{1}{2}+\frac{\gamma}{2}-\frac{\delta}{4}}  \log \left( N\left( |h_1| + |h_2| + \cdots + |h_k| \right)\right) \log( N) \right),
    \end{equation*}
    where the implied constant only depends on $\pmb \alpha$.
\end{proposition}

The proof of this proposition follows an approach similar to that of Murty and Srinivas \cite{ram}. We recall \cite[Lemma 1, Lemma 2]{ram}  below, and refer the reader to their paper for the detailed proofs.
\medskip

\begin{lemma}[Effective Poisson summation formula]\label{lemma-1}
Let $f(t)$ be a differentiable function on $[1,N]$ satisfying $|f'(t)|\leq K$. Then, for any $M\geq 1$,
\begin{equation*}
    \sum_{j=1}^N f(j) = \sum_{0\leq |m| \leq M} \int_{1}^{N} f(t)\, e(mt) \,dt + O\left(\frac{NK \log M}{M}\right).
\end{equation*}
\end{lemma}

\medskip

\begin{lemma}\label{lemma-2}
Let $F(x)$ be real, twice differentiable function on $[a,b]$ such that $|F''(x)| \geq m > 0$. Then,
\begin{equation*}
    \left| \int_a^b e^{iF(x)} \,dx \right| \leq \frac{8}{\sqrt{m}}.
\end{equation*}
\end{lemma}
\medskip

\begin{proof}[Proof of Proposition \ref{unif-distribution}] 
Let $0<\alpha_1,\alpha_2,\ldots,\alpha_k<1$, $\gamma = \max(\alpha_j)$ and $\delta = \min(\alpha_j)$. Consider the dyadic sum
\begin{equation*}
    \rS(W,2W) := \sum_{W\leq n< 2W} e\left(h_1n^{\alpha_1}+h_2n^{\alpha_2} +\cdots+h_kn^{\alpha_k} \right).
\end{equation*}
We will take $f(t) =   e\left(h_1t^{\alpha_1}+h_2t^{\alpha_2} +\cdots+h_k t^{\alpha_k} \right)$ in Lemma \ref{lemma-1}. Note that
\begin{align*}
    |f'(t)| &= 2\pi \left| h_1\alpha_1t^{\alpha_1-1} + h_2\alpha_2t^{\alpha_2-1} +\cdots+ h_k\alpha_k t^{\alpha_k-1}\right| \\ 
    &\leq 2\pi \left(\alpha_1\left| h_1 \right| t^{\alpha_1-1} +\alpha_2 \left| h_2 \right|t^{\alpha_2-1} +\cdots+ \alpha_k \left| h_k \right|t^{\alpha_k-1}\right)\\
    & \ll \left(|h_1| + |h_2| + \cdots + |h_k|\right) t^{\gamma-1}.
\end{align*}
Therefore, in the interval $[W,2W]$,
\begin{equation*}
    |f'(t)| \ll \left(|h_1| + |h_2| + \cdots + |h_k|\right)\, W^{\gamma-1},
\end{equation*}
where the implied constant is absolute. Thus, by Lemma \ref{lemma-1}, we get
\begin{align}
    \rS(W,2W) = \sum_{0\leq |m| \leq M} \int_W^{2W} & e(h_1t^{\alpha_1}+h_2t^{\alpha_2} +\cdots+h_k t^{\alpha_k} + mt)  \, dt \nonumber \\
    & + O\left(\frac{\left(|h_1| + |h_2| +\cdots + |h_k| \right)\, W^{\gamma} \,\log M}{M}\right),\label{equation-1}
\end{align}
where $M$ is some large positive constant. To bound the integral in the sum above, we take 
\begin{equation*}
    F(t) = h_1t^{\alpha_1} +h_2t^{\alpha_2} +\cdots+h_kt^{\alpha_k} +  mt
\end{equation*}
in Lemma \ref{lemma-2}. Then
\begin{equation*}
    F''(t) = \alpha_1 (\alpha_1 -1) t^{\alpha_1-2}h_1 + \alpha_2 (\alpha_2 -1) t^{\alpha_2-2}h_2 +\cdots+ \alpha_k (\alpha_k -1) t^{\alpha_k-2}h_k.
\end{equation*}
The conditions imposed on $h_i$'s ensure that $F''(t) \neq 0$. Since $0<\alpha_i<1$, for $t\in [W,2W]$, we obtain
\begin{equation*}
    |F''(t)| \gg \,\left| h_1 + h_2 + \cdots+h_k\right| \, W^{\delta -2} >0,
\end{equation*}
where the implied constant depends on $\alpha_1,\alpha_2,\ldots,\alpha_k$. Hence, from Lemma \ref{lemma-2}, we obtain
\begin{align*}
    \left| \int_W^{2W} e\left(h_1t^{\alpha} +h_2t^{\alpha_2} +\cdots+h_kt^{\alpha_k}+ mt\right)\, dt \right| = O\left( \,\left( \left|h_1 + h_2 + \cdots + h_k\right| \right)^{-\frac{1}{2}}W^{1-\delta/2} \right)\\
    = O\left( \,\left( |h_1| + |h_2| + \cdots + |h_k| \right)^{-\frac{1}{2}}W^{1-\delta/2} \right).
\end{align*}

\noindent
Therefore, from \eqref{equation-1}, we get
\begin{equation*}
    \rS(W,2W) = O\left(\frac{\left(|h_1| + |h_2| + \cdots + |h_k| \right)\, W^{\gamma} \,\log M}{M}\right) + O\left( M\left( |h_1| + |h_2| + \cdots + |h_k| \right)^{-\frac{1}{2}}W^{1-\delta/2} \right).
\end{equation*}
Now, choosing $M= \floor{ C_2\left(  |h_1| + |h_2|+ \cdots+ |h_k|\right)^{\frac{3}{4}} W^{-\frac{1}{2}+\frac{\gamma}{2}+\frac{\delta}{4}}}$ for a sufficiently large constant $C_2$, we obtain
\begin{align*}
    \rS(W,2W) = O \left( \left( |h_1| + |h_2| + \cdots + |h_k| \right)^\frac{1}{4}W^{\frac{1}{2}+\frac{\gamma}{2}-\frac{\delta}{4}}  \log \left( W\left( | h_1| + |h_2| + \cdots + |h_k|\right)\right) \right).
\end{align*}
Since the interval $[1,N]$ can be split into $O(\log N)$ dyadic intervals of the type $[W,2W]$, we get
\begin{equation*}
    \rS(N, \mathbf{h}, \pmb\alpha)= O\left(\left( |h_1| + |h_2| + \cdots + |h_k| \right)^\frac{1}{4}N^{\frac{1}{2}+\frac{\gamma}{2}-\frac{\delta}{4}}  \log \left( N\left( |h_1| + |h_2| + \cdots + |h_k| \right)\right) (\log N) \right).
\end{equation*}
This proves Proposition \ref{unif-distribution}.
\end{proof}
\bigskip

We also make the following easy observation.

\begin{lemma}\label{greatest-integer-congruence}
    Let $x$ be a positive real number and $d$ be a natural number $\geq 2$. Then,
    \begin{equation*}
        \floor{x}\equiv r \bmod d \iff \frac{r}{d} \leq \left\{\frac{x}{d}\right\} < \frac{r+1}{d}.
    \end{equation*}
\end{lemma}
\begin{proof}
Write $x = \floor{x} + \delta$. If $\floor{x}\equiv r \bmod d$, then $\floor{x} = dt + r$ for some integer $t$. Hence, 
\begin{equation*}
    \frac{x}{d} = \frac{\floor{x}}{d} + \frac{\{x\}}{d}= t + \frac{r}{d} + \frac{\delta}{d}.
\end{equation*}
Since $0\leq\delta <1$, we get $r/d \leq \{x/d\} < (r+1)/d$. Conversely, write $x = dt + a + \{x\}$, where $0\leq a  < d$ is a positive integer. Since, $r/d \leq \{x/d\} < (r+1)/d$, we clearly have $a=r$.
\end{proof}
\medskip

Recall the Erd\H{o}s-Tur\'{a}n-Koksma Theorem (see \cite[p. 116]{Kuipers}) on discrepancy of sequences. Let $P=\{\mathbf{x}_1,\mathbf{x}_2,\ldots,\mathbf{x}_N\}$ be a finite sequence in $\mathbb{R}^{k}$. For a subinterval $J\subseteq [0,1)^k$ of the form $J=[a_1,b_1)\times [a_2,b_2)\times\cdots \times [a_k,b_k)$, define
\begin{equation*}
    A(J;P):= \#\left\{ \mathbf{x}_i \in P: \mathbf{x}_i\in J\right\}.
\end{equation*}
The discrepancy $D_N$ for $P$ is defined by 
\begin{equation*} 
    D_N := \sup_{J} \left| \frac{A(J;P)}{N} - \lambda(J) \right|,
\end{equation*}
where   $J$ runs through all subintervals of $[0,1)^k$ of the above form and $\lambda$ denotes the $k$-dimensional Lebesgue measure. We now state a version of the Erd\H{o}s-Tur\'{a}n-Koksma theorem, which was independently proved by F. J. Koksma \cite{koksma} and P. Sz\"{u}zu \cite{szusz}.

\begin{theorem}[Erd\H{o}s-Tur\'{a}n-Koksma]\label{erdos-turan}
For a lattice point $\boldsymbol\ell = (l_1,\ldots,l_k)$ in $\mathbb{Z}^k$, define 
\begin{equation*}
    \|\boldsymbol\ell\|_{\infty} = \max_{1\le j \le k} |l_j| \qquad \text{ and} \qquad  r(\boldsymbol\ell) = \prod_{j=1}^{k} \max(|l_j|,1).
\end{equation*}
For $\mathbf{x},\mathbf{y} \in \mathbb{R}^k$, let $\langle \mathbf{x},\mathbf{y}\rangle$ denote the standard inner product. Let $\mathbf{x}_1,\ldots,\mathbf{x}_N$ be a finite sequence of points in $\mathbb{R}^k$. Then, for any positive integer $m$, we have \\

\begin{equation*}
    D_N \leq 2k^2 3^{k+1} \left(\frac{1}{m} + \sum_{0<\|\boldsymbol\ell\|_{\infty} \le m} \frac{1}{r(\boldsymbol\ell)} \,\left| \frac{1}{N}\sum_{n=1}^N e(\langle\boldsymbol\ell,\mathbf{x}_n\rangle)\right|\right).
\end{equation*}
\end{theorem}
\bigskip

\section{\bf Proof of the main theorems}
\medskip

\begin{proof}[Proof of Theorem \ref{multivariable-equidistribution}] By Lemma \ref{greatest-integer-congruence}, we can write
\begin{align*}
    S(N, \mathbf{d}, \mathbf{c}, \pmb \alpha) &= \#\left\{ n\leq N: \floor{n^{\alpha_i}} \equiv c_i \bmod d_i \,\, \text{for all}\,\, 1\leq i\leq k \right\}\\
    &=\#\left\{n\leq N : \frac{c_i}{d_i} \leq \left\{\frac{n^{\alpha_i}}{d_i}\right\}< \frac{c_{i}+1}{d_{i}}\,\, \text{for all}\,\, 1\leq i \leq k \right\}.
\end{align*}
Define 
$$
\mathbf{x}_n= \left(\left\{\frac{n^{\alpha_1}}{d_1}\right\}, \left\{\frac{n^{\alpha_2}}{d_2}\right\}, \ldots, \left\{\frac{n^{\alpha_k}}{d_k}\right\}\right).
$$ 
Let $D_N$ denote the discrepancy of the sequence $P=\{\mathbf{x}_1,\mathbf{x}_2,\ldots, \mathbf{x}_N\}$. Then, 
\begin{equation}\label{discrepancy-relation}
    ND_N \geq \sup\left|\#\left\{n\leq N :  \left\{\frac{n^{\alpha_i}}{d_i} \right\} \in \bigg[\frac{c_i}{d_i}, \frac{c_i + 1}{d_i}\bigg) \text{ for all } 1\leq i\leq k\right\}  - \frac{N}{d_1d_2\ldots d_k}\right|.
\end{equation}
Applying Theorem \ref{erdos-turan}, for any $m\geq 1$,
\begin{equation*}
    ND_N \leq 2k^23^{k+1} \left( \frac{N}{m} + \sum_{0<\|\boldsymbol\ell\|_{\infty} \leq m} \frac{1}{r(\boldsymbol\ell)} \,\left| \sum_{n=1}^N e(\langle\boldsymbol\ell,\mathbf{x}_n\rangle)\right|\right),
\end{equation*}
where $\boldsymbol\ell$ runs over integer lattice points. Note that
\begin{equation*}
    e\left( \langle\boldsymbol\ell,\mathbf{x}_n\rangle \right) = e\left(\sum_{i=1}^k l_i \left\{\frac{n^{\alpha_i}}{d_i}\right\}\right) = e\left(\sum_{i=1}^{k} \frac{l_i}{d_i} \, n^{\alpha_i} \right).
\end{equation*}
Since $\alpha_i$'s are distinct in $(0,1)$, the condition in Proposition \ref{unif-distribution} is satisfied for all non-zero lattice points $\boldsymbol\ell$. Hence, we deduce that
\begin{equation}\label{discrepancy}
    ND_N = O\left(\frac{N}{m}\right) + O\left(N^{\frac{1}{2}+\frac{\gamma}{2}-\frac{\delta}{4}} \, \, \log N \, \sum_{0<\|\boldsymbol\ell\|_{\infty}\leq m} \frac{ \left(  \sum_{1\leq i\leq k} \frac{l_i}{d_i}  \right)^{\frac{1}{4}} }   {\prod_{1\leq i\leq k} \max(l_i,1)} \,\,\log \left( N \sum_{1\leq i\leq k} \frac{l_i}{d_i}   \right) \right).
\end{equation}
The inner sum can be bounded above as
\begin{align*}
    \sum_{0<\|\boldsymbol\ell\|_{\infty}\leq m} \frac{ \left(  \sum_{1\leq i\leq k} \frac{l_i}{d_i}  \right)^{\frac{1}{4}} }   {\prod_{1\leq i\leq k} \max(l_i,1)} \,\,\log \left( N \sum_{1\leq i\leq k} \frac{l_i}{d_i}   \right) & \ll \frac{1}{\min(d_i)^{1/4}} \, \sum_{\substack{0<l_1 = \|\boldsymbol\ell\|_{\infty}\leq m\\\\ l_i\leq l_1}} \frac{l_1^{-3/4} (\log Nl_1)}{\prod_{2\leq i\leq k} \max(l_i,1)}\\
    &\ll \frac{1}{\min(d_i)^{1/4}}  \sum_{0<l_1\leq m} l_1^{-3/4} (\log Nl_1) \prod_{2\leq i\leq k}\left(\sum_{l_i\leq l_1} \frac{1}{\max(1,l_i)}\right)\\
    &\ll \frac{m^{1/4} (\log N) (\log m)^{k-1} }{\min(d_i)^{1/4}},
\end{align*}
where the implied constant depends on $k$. Therefore, from \eqref{discrepancy}, we obtain
\begin{equation*}
    ND_N = O\left(\frac{N}{m}\right) + O\left(\frac{N^{\frac{1}{2}+\frac{\gamma}{2}-\frac{\delta}{4}} (\log N)^2 m^{\frac{1}{4}} (\log m)^{k-1}}{(\min(d_i))^\frac{1}{4}}  \right)
\end{equation*}
Now, choosing $m = \floor{N^{\frac{\delta}{5}-\frac{2\gamma}{5}+\frac{2}{5}}}$, we  have
\begin{equation*}
    ND_N = O\left(\frac{N^{\frac{3 + 2\gamma - \delta}{5}} (\log N)^{k+1}}{(\min(d_i))^\frac{1}{4}} \right).
\end{equation*}
Hence, from \eqref{discrepancy-relation}, we conclude Theorem \ref{multivariable-equidistribution}.
\end{proof}
\medskip

\begin{proof}[Proof of Theorem \ref{primes-simultaneous}] 
Recall that
\begin{equation*}
    \Lambda(n) = -\sum_{d\mid n} \mu(d)\log d.
\end{equation*}
We can write
\begin{align*}
    \sum_{n\leq N} \Lambda(\floor{n^{\alpha_1}} + c_1)\Lambda(\floor{n^{\alpha_2}}+c_2)&\ldots \Lambda(\floor{n^{\alpha_k}}+c_k) \\
    &= \sum_{n\leq N} \left(-\sum_{d_1 \mid \, \floor{n^{\alpha_1}}+c_1}\mu(d_1)\, \log d_1\right)\cdots \left(-\sum_{d_k \mid \, \floor{n^{\alpha_k}}+c_k}\mu(d_k)\, \log d_k\right). 
\end{align*}
Interchanging summation, the above sum equals
\begin{equation*}
     \sum_{d_1 \leq N^{\alpha_1}+c_1}\cdots\sum_{d_k \leq N^{\alpha_k}+c_k}(-\mu(d_1)\log d_1)\cdots(-\mu(d_k)\log d_k)\sum_{\substack{n \leq N \\ d_i \mid \,\floor{n^{\alpha_i}}+c_i}} 1
\end{equation*}
The inner sum is nothing but $S(N, \mathbf{d},\mathbf{c}, \pmb{\alpha})$. Thus, applying  Theorem \ref{multivariable-equidistribution}, we obtain
\begin{align*}
& \sum_{n\leq N} \Lambda(\floor{n^{\alpha_1}}+c_1) \cdots \Lambda(\floor{n^{\alpha_k}}+c_k) \\
& = N \left(\sum_{d_1 \leq N^{\alpha_1}+c_1}\frac{\mu(d_1)\, \log d_1}{d_1}\,\,\cdots\sum_{d_k \leq N^{\alpha_k}+c_k}\frac{\mu(d_k)\, \log d_k}{d_k} \right) \\
& \hspace{2cm} +O\left( \sum_{d_1 \leq N^{\alpha_1}+c_1}\cdots\sum_{d_k \leq N^{\alpha_k}+c_k} \frac{ (\log d_1)\ldots (\log d_k)}{\min(d_j)^\frac{1}{4}}N^{\frac{3+2\gamma-\delta}{5}}(\log(N)^{k+1})\right).\\
\end{align*}
It is easy to see that
\begin{align*}
    \sum_{d_1 \leq N^{\alpha_1}+c_1}\cdots\sum_{d_k \leq N^{\alpha_k}+c_k} \frac{ (\log d_1)\ldots (\log d_k)}{\min(d_j)^\frac{1}{4}} \ll N^{(\sum_{i=1}^k \alpha_i) - \frac{\delta}{4}}(\log(N))^k.
\end{align*}
Also, recall the following well known identity (see \cite[p. 596-597]{landau}),
\begin{equation*}
    -\sum_{n\leq x} \mu(n) \frac{\log n}{n} = 1 + O\left(e^{-C\sqrt{\log x}}\right)
\end{equation*}
for some $C>0$. Therefore, we get
\begin{align*}
& \sum_{n\leq N} \Lambda(\floor{n^{\alpha_1}}+c_1) \cdots \Lambda(\floor{n^{\alpha_k}}+c_k) \\
&=  N \left( 1 + O(e^{-C\sqrt{ \log(N)}})\right)  + O \left( N^{\frac{3+2\gamma}{5}-\frac{9\delta}{20}+ \sum_{i=1}^k \alpha_i }\,\,(\log(N)^{2k+1}) \right),
\end{align*}
which proves Theorem \ref{primes-simultaneous}.
\end{proof}
\bigskip

\section{\bf Applications to other arithmetic functions}
\bigskip

The equidistribution result in Theorem \ref{multivariable-equidistribution} can be applied to a variety of arithmetic functions to derive interesting consequences. One such application concerns squarefreeness of the sequence $\floor{n^{\alpha_i}}$. More precisely: Given real numbers $\alpha_1,\alpha_2,\ldots, \alpha_k$, how often are $\floor{n^{\alpha_i}}$ all squarefree numbers? Since $\mu^2(n)$ is the characteristic function for squarefree numbers, this is equivalent to evaluating the sum
\begin{equation*}
    \sum_{n\leq N} \mu^2(\floor{n^{\alpha_1}})\cdots \mu^2(\floor{n^{\alpha_k}}).
\end{equation*}
By the identity
\begin{equation*}
    \mu^2(n) = \sum_{d^2|n} \mu(d),
\end{equation*}
we can interchange the summation to obtain
\begin{align*}
    \sum_{n\leq N} \mu^2(\floor{n^{\alpha_1}})\cdots \mu^2(\floor{n^{\alpha_k}}) &= \sum_{n\leq N} \left(\sum_{d_1^2\mid \,\floor{n^{\alpha_1}}}\mu(d_1)\right)\left(\sum_{d_2^2\mid \,\floor{n^{\alpha_2}}}\mu(d_2)\right)\cdots \left(\sum_{d_k^2\mid \,\floor{n^{\alpha_k}}}\mu(d_k)\right)\\
    &=\sum_{d_1\leq N^{\alpha_1/2}} \cdots \sum_{d_k\leq N^{\alpha_k/2}} \mu(d_1)\mu(d_2)\ldots \mu(d_k) \sum_{\substack{n\leq N\\ d_i^2 \mid \, \floor{n^{\alpha_i}}}} 1.
\end{align*}
\noindent The inner sum is nothing but $S(N, \mathbf{d}, \mathbf{0}, \pmb{\alpha})$, where $\mathbf{d}=(d_1^2,d_2^2,\ldots, d_k^2)$. Now, using Theorem \ref{multivariable-equidistribution}, for $0<\delta=\alpha_1<\cdots<\alpha_k=\gamma<1$, we conclude that
\begin{align*}
    \sum_{n\leq N} \mu^2(\floor{n^{\alpha_1}})\cdots \mu^2(\floor{n^{\alpha_k}}) & = N\left( \sum_{d_1\leq N^{\alpha_1/2}} \frac{\mu(d_1)}{d_1^2}\right)\cdots \left(\sum_{d_k\leq N^{\alpha_k/2}} \frac{\mu(d_k)}{d_k^2}\right)\\ 
    &+ O\left(\sum_{d_i\leq N^{\alpha_i/2}} \frac{N^{\frac{3+2\gamma-\delta}{5}} (\log N)^{k+1}}{(\min(d_i^2))^\frac{1}{4}}\right).
\end{align*}
Since
\begin{equation*}
    \sum_{n\leq x} \frac{\mu(n)}{n^2} = \frac{1}{\zeta(2)} + O(x^{-1}),
\end{equation*}
we obtain
\begin{equation*}
    \sum_{n\leq N}  \mu^2(\floor{n^{\alpha_1}})\cdots \mu^2(\floor{n^{\alpha_k}}) = \frac{N}{\zeta(2)^k} + O \left( N^{(\sum_{i=1}^k \frac{\alpha_i}{2}) -\frac{\delta}{4} +\frac{3+2\gamma-\delta}{5}} \right),
\end{equation*}
and hence, deduce the following theorem.
\begin{theorem}\label{squarefree}
    Let $0<\delta= \alpha_1<\cdots<\alpha_k=\gamma <1$. Suppose
    $$ 
    \sum_{i=1}^k \alpha_i < \frac{4-4\gamma}{5} + \frac{9\delta}{20}.
    $$ 
    Then there are infinitely many positive integers $n$ such that $\floor{n^{\alpha_i}}$ are all squarefree and the number of such $n\leq x$ is given by
    \begin{align*}
        \sum_{n\leq N} \mu^2(\floor{n^{\alpha_1}})\cdots \mu^2(\floor{n^{\alpha_k}}) = \frac{N}{\zeta(2)^k} + O\left(N^{\max\left( 1-\frac{\delta}{2}, \, \, (\sum_{i=1}^k \frac{\alpha_i}{2})+\frac{3+2\gamma}{5}-\frac{9\delta}{20} \right)} \, (\log N)^{k+1}\right).
    \end{align*}
    
\end{theorem}
\medskip

Similar method can also be used to compute the sum of other arithmetic functions over $\floor{n^{\alpha_i}}$. For instance, denote the divisor function $d(n) = \sum_{d|n} 1$ and the sum of divisors function $\sigma(n)=\sum_{d|n} d$. Then, as a consequence of Theorem \ref{multivariable-equidistribution}, for $0<\delta= \alpha_1<\cdots<\alpha_k=\gamma <1$, we can prove the following.
\begin{equation*}
    \sum_{n\leq N} d(\floor{n^{\alpha_1}})\cdots d(\floor{n^{\alpha_k}}) =  (\alpha_1\alpha_2\cdots \alpha_k)\,N(\log N)^k + O\left( N^{\frac{3+2\gamma}{5}-\frac{9\delta}{20} + \sum_{i=1}^k\alpha_i} \,\,(\log(N))^{k+1}\right)
\end{equation*}
and 

\begin{equation*}
    \sum_{n\leq N} \sigma(\floor{n^{\alpha_1}})\cdots \sigma(\floor{n^{\alpha_k}}) = N^{1+\alpha_1+\alpha_2+\cdots+\alpha_k} \,+ \,  O\left(N^{\frac{3+2\gamma}{5}-\frac{9\delta}{20} + 2\sum_{i=1}^k \alpha_i} \,\,(\log(N))^{k+1}\right).
\end{equation*}

\bigskip

\section{\bf Concluding remarks}
\bigskip

Note that Theorem \ref{primes-simultaneous} is conditional on $\alpha_i$'s satisfying $\sum_i \alpha_i \leq \frac{9\delta}{20} - \frac{2\gamma}{5}+ \frac{2}{5}$. This seems to be a limitation of the method used in this paper. It is reasonable to conjecture that if $0<\alpha_1<\cdots<\alpha_k<1$, then for infinitely many $n$, all of $\floor{n^{\alpha_i}}$ are simultaneously primes. If one hopes to use the method in this paper, one has to prove a sharper error term in Theorem \ref{multivariable-equidistribution}.  In the light of Piatetski-Shapiro primes, it is tempting to conjecture that there exists a constant $\lambda>1$ such that whenever $0<\alpha_1<\cdots< \alpha_k<\lambda$, the numbers $\floor{n^{\alpha_i}}$ are simultaneously primes for infinitely many $n$. We relegate these questions to future investigation.

\bigskip

\section*{\bf Acknowledgments}
We are grateful to Prof. Ram Murty for introducing the theme of uniform distribution, sharing \cite{ram} and for several fruitful discussions. 
The second author would like to thank the Institute of Mathematical Sciences for hospitality during the summer research program in 2025.

\end{document}